\documentclass[12pt]{amsart}

\usepackage{amsthm, amsmath, amssymb, fullpage, bbm, amsrefs}
\usepackage{hyperref}
\usepackage{color}

\newcommand{\T}{\mathbb{T}} 
\newcommand{\C}{\mathbb{C}} 
\newcommand{\D}{\mathbb{D}} 
\newcommand{\N}{\mathbb{N}}
\newcommand{\cd}{\overline{\D}} 
\newcommand{\al}{\alpha}

\newcommand{\Sz}{\mathcal{S}}
\newcommand{\mA}{\mathcal{A}}
\newcommand{\mH}{\mathcal{H}}

\numberwithin{equation}{section}

\newtheorem{theorem}{Theorem}[section]

\newtheorem{corollary}[theorem]{Corollary}
\newtheorem{lemma}[theorem]{Lemma}

\theoremstyle{definition}

\title{Three radii associated to Schur functions on the polydisk}

\author{Greg Knese}
\address{Department of Mathematics, Washington University in St Louis, St Louis, MO 63130, USA}
\email{geknese@wustl.edu}

\date{\today}

\thanks{Partially supported by NSF grant DMS-2247702}



\begin{document}

\begin{abstract}
This article examines three radii associated to bounded analytic
functions on the polydisk: the well-known Bohr radius, the Bohr-Agler radius,
and the Schur-Agler radius.  We prove explicit upper and lower bounds for
the Bohr-Agler radius, an explicit lower bound for the Schur-Agler radius,
and an asymptotic upper bound for the Schur-Agler radius.  
The Bohr-Agler radius obeys the same (known) asymptotic as
the Bohr radius while we show the Schur-Agler radius
is roughly of the same growth as the Bohr radius.
As a corollary, we bound the Bohr radius on the bidisk below by $0.3006$.
Finally, we improve some estimates of P.G.\ Dixon on Agler norms
of homogeneous polynomials using some modern inequalities.
\end{abstract}

\maketitle

\tableofcontents

\section{Introduction}

In this article we examine three radii associated to functions
analytic on the unit polydisk $\D^d =\{z =(z_1,\dots, z_d) \in \C^d: |z_1|,\dots, |z_d|<1\}$.
In particular, we study the relationship between the well-known 
Bohr radius and two lesser studied radii which we
refer to as the Bohr-Agler radius and the Schur-Agler radius.  
Our purpose is to gain some insight into whether operator inequalities
 for analytic functions are closely related to absolute convergence
or cancellation properties of the functions.  

First, the \emph{Bohr radius} is a way of measuring the absolute convergence or cancellation
properties of the power series of bounded analytic functions.
The Schur class $\Sz_d$ of the polydisk $\D^d$ 
is the set of all analytic $f:\D^d\to \cd$.
Given $f \in \Sz_d$ with power series $f(z) = \sum_{\al} f_{\al} z^{\al}$
we define the coefficient-wise $\ell^1$ norm
\begin{equation} \label{ell1}
\|f\|_1 := \sum_{\al} |f_{\al}|.
\end{equation}
(A more precise notation might be $\|\hat{f}\|_{\ell^1}$ but
we write the above for simplicity.)
Setting $f_r(z) = f(rz)$ for $0<r<1$,
 the Bohr radius of $\D^d$ is defined to be
\[
K_d = \sup\{r>0: \text{ for all } f \in \mathcal{S}_d \text{ we have } \|f_r\|_1 \leq 1 \}.
\]
The Bohr radius was introduced by Harald Bohr in the study of Dirichlet series
and he proved $K_1 = 1/3$ \cite{Bohr}.
Boas and Khavinson \cite{BK} proved
\[
\frac{1}{3\sqrt{d}} \leq K_d \leq 2 \sqrt{\frac{\log d}{d}}.
\]
After the culmination of deep work by many authors
the precise asymptotic 
\[
K_d \sim \sqrt{\frac{\log d}{d}}
\]
 was established; 
see \cite{defant}, \cite{bayart}.  
Boas-Khavinson remarked an improved \emph{explicit} lower bound by stating that
$K_d$ is at least the value of the root $r$ of the equation
\[
r + \sum_{k=2}^{\infty} r^k \binom{d+k-1}{k}^{1/2} = \frac{1}{2}.
\]
For instance, when $d=2$
\[
r+ \sum_{k=2}^{\infty} r^k \sqrt{k+1} = \frac{1}{2}.
\]
This can be computed via its relation to the polylogarithm
\[
Li_{s}(z) = \sum_{k=1}^{\infty} \frac{z^k}{k^{s}}
\]
and in this case we are solving
\[
r + \frac{1}{r}(Li_{-1/2}(r) -r - r^2\sqrt{2}) = 1/2 \text{ or rather }
Li_{-1/2}(r) = \frac{3}{2}r + (\sqrt{2}-1)r^2.
\]
Computer software yields $K_2 \geq 0.287347$. 
It still seems that the only known value
of $K_d$ is $K_1 = 1/3$.  Also, we do not know of any further explicit
bounds on $K_d$ beyond the Boas-Khavinson inequality and the 
bound $K_2 \geq 0.3006$ that we obtain in this paper (see Corollary \ref{bidisk}).
All of the other
bounds are asymptotic or involve an unnamed constant.
The paper \cite{newapproach} proves a sharper lower bound based
on power series whose coefficients are certain Sidon constants---which are
not explicitly known in any simple closed form.

Our next two radii are about an interesting subclass of the Schur class,
namely the \emph{Schur-Agler class}.
The Schur-Agler class of the polydisk, $\mA_d$, is the set of
$f \in \Sz_d$ which satisfy
\[
\|f(T_1,\dots, T_d)\| \leq 1
\]
for all $d$-tuples $T=(T_1,\dots, T_d)$ of commuting strictly contractive operators
on a (common) Hilbert space.  We
will call this the Agler class for brevity.  
Functions in the Agler class possess a variety of features
that are lacking for general Schur class functions.
Most notable for us is the transfer function realization formula
used in Section \ref{sec:KAlower}.  Part of our motivation
is not just to study the Agler class but to see how close Schur class
functions are to belonging to the Agler class.
Von Neumann's inequality
proves that $\Sz_1 = \mA_1$ \cite{vN} and And\^{o}'s inequality proves 
$\Sz_2 = \mA_2$ \cite{Ando}.  Counterexamples first constructed by Varopoulos
demonstrate that $\Sz_d \ne \mA_d$ for $d>2$ \cite{varo}.  
For a variety of interesting recent papers about the Agler class see
the references \cite{Anderson}, \cite{Bhowmik}, \cite{Barik}, \cite{Kojin}, \cite{Bickel}, \cite{Debnath}.

One can consider the Bohr radius of the Agler class,
which might be called the \emph{Bohr-Agler radius}:
\[
K(\mA_d) := \sup \{ r>0: \|f_r\|_1 \leq 1 \text{ for all } f \in \mA_d\}.
\]
Necessarily, $K(\mA_d) \geq K_d$ because our requirement
is over a smaller class of functions.  
Also, $K(\mA_1) = K_1, K(\mA_2) = K_2$.

Finally, we consider the \emph{Schur-Agler radius}
\[
SA_d = \sup\{r>0: f_r \in \mA_d \text{ for all } f \in \Sz_d\}.
\]
Necessarily,  $SA_1=SA_2 = 1$ and $SA_d \geq K_d$.  Indeed,
if $r < K_d$ then for $f \in \Sz_d$ we have
 $\sum_{\al} r^{|\al|} |f_{\al}| \leq 1$
 and then we also have for any $d$-tuple of strict commuting
 contractions $T$
 \[
 \|f(rT)\| \leq \sum_{\al} \|T^{\al}\| r^{|\al|}|f_{\al}| \leq 1
 \]
 and therefore $r\leq SA_d$.  Taking $r\nearrow K_d$ we have $K_d \leq SA_d$.
 
While the Agler class does not have any direct connections to absolute convergence
beyond the simple inequality just described, the methods used in the study of
the Agler class and the Bohr radius have some remarkable similarities.
In both cases random methods are used to build 
exceptional examples based on 
the Kahane-Salem-Zygmund inequality 
(see \cite{BK}, \cite{varo}, \cite{Kahane} Chapter 6, \cite{dixon}).
In addition, it is of interest to know whether membership in the Agler class
implies any special absolute convergence properties.  

Our main results will be some explicit inequalities for the above radii as
well as some asymptotic estimates.  Our work has a few novel ingredients 
as well as a synthesis of a variety of results in the literature.

The first result is an explicit estimate for $K(\mA_d)$.

\begin{theorem}\label{KAlower}
For $d\geq 1$
\[
K(\mA_d) \geq \frac{1}{\sqrt{d}+2}.
\]
Even more, for $d \geq 1$, 
$K(\mA_d)$ is greater than or equal to the 
root $r>0$ of the equation
\[
\sum_{k=1}^{\infty} r^k \binom{d+k-2}{k-1}^{1/2} = 1/2.
\]
\end{theorem}

Since $K(\mA_d) \geq K_d \sim \sqrt{\frac{\log d}{d}}$
the estimate above is obviously not optimal for large $d$.
However, it could be useful for small $d$.
In the case $d=2$, we must solve
\[
\sum_{k=1}^{\infty} r^k \sqrt{k} = Li_{-1/2}(r) = 1/2
\]
to see that $r\geq 0.3006$.
\begin{corollary} \label{bidisk}
\[
K(\mA_2) = K_2 \geq 0.3006.
\]
\end{corollary}

This is the best lower bound on $K_2$ we have seen in the literature.  We do not
know of any non-trivial upper bounds on $K_2$ beyond $K_1=1/3$.  

Next, we obtain an explicit lower bound for $SA_d$.

\begin{theorem} \label{SAlower}
For $d \geq 2$
\[
SA_d \geq \frac{1}{\sqrt{d-1}}
\]
\end{theorem}

This is a minor improvement over a result of Brehmer which proves the
lower bound $\frac{1}{\sqrt{d}}$ \cite{Brehmer}; this is
explicitly stated in \cite{lubin}.  Theorem \ref{SAlower}
uses a result of Grinshpan et al.\ \cite{grinshpan} in a fundamental way.
Again, the result is not of the right growth order since $SA_d \geq K_d$.

Using a construction of Maurizi-Queffelec \cite{MQ} originally designed to show
optimality of the exponents in the famous Bohnenblust-Hille inequalities \cite{BH},
we are able to give an upper bound for $K(\mA_d)$ because the examples
constructed there conveniently belong to the Agler class.

\begin{theorem} \label{KAupper}
For positive integers $q,m$
\[
q^{\frac{1}{2m}-\frac{1}{2}} \geq K(\mA_{qm})
\]
which leads to the {\color{red} estimate 
\[
 K(\mA_d) \leq e^{1/2} \sqrt{\frac{\log d }{d}}\exp\left(O\left(\frac{1}{\log d}\right)\right).
 \]
 and
 \[
1 \leq \liminf_{d\to \infty} \frac{K(\mA_d)}{\sqrt{\frac{\log d }{d}}} 
\leq \limsup_{d\to \infty} \frac{K(\mA_d)}{\sqrt{\frac{\log d }{d}}} \leq e^{1/2}.
\] 
 }
\end{theorem}
 
Work in 
Dixon \cite{dixon} allows us to 
obtain an asymptotic upper bound for $SA_d$. 

\begin{theorem} \label{SAupper}
   There exists a constant $c>0$ such that
  \[
  SA_d \leq c \left(\frac{\log d}{d}\right)^{1/4}.
  \]
  \end{theorem}
  
Thus, there is a gap between the lower estimate $K_d \sim \sqrt{\frac{\log d}{d}}$ and the upper estimate above.
  
In addition, the work 
of Dixon \cite{dixon} also obtains useful inequalities for comparing the Schur norm
\[
\|f\|_\infty = \sup\{ |f(z)| : z \in \D^d\}
\]
and the Agler norm
\begin{equation} \label{Aglernorm}
\|f\|_{A} = \sup\{ \|f(T)\|\}
\end{equation}
where the supremum is taken over all $d$-tuples $T$ of commuting strictly
contractive operators on a Hilbert space.
It is tempting to try to use some of Dixon's inequalities to prove
a lower bound for $SA_d$, however we show that these inequalities
with modern improvements incorporated
cannot beat the bound in Theorem \ref{SAlower}.  
Defining
\begin{equation} \label{Ckd}
C_{k}(d) = \sup\left\{ \frac{\|p\|_{A}}{\|p\|_{\infty}}: p\in \C[z_1,\dots, z_d] \text{ homogeneous and } \deg p = k\right\}
\end{equation}
we obtain the following estimate by inserting an
estimate from \cite{bayart} and the polarization inequality from \cite{harris}.

\begin{theorem} \label{improvedixon}
There exists a constant $c$ such that for $d,k\geq 2$
\[
C_k(d) \leq c k  (e/2)^{k} d^{\frac{k-2}{2}}.
\]
\end{theorem}

The proof is in Section \ref{sec:dixon}. As far as we can
tell this estimate is only sufficient for proving 
$SA_d \geq c \frac{1}{\sqrt{d}}$, however our rehashing
of the proof and demonstration of the constants
involved may inspire others to do better.
In particular, bridging the methods of Dixon and
those of Hartz \cite{hartz}, discussed at the end of the paper,
may prove fruitful.

\section*{Acknowledgements}
Thank you to Michael Hartz for some valuable comments on a draft of this paper.
Thank you to the anonymous referee for a careful reading and several very useful corrections
and suggestions.

\section{Lower bound for Bohr-Agler radius: Theorem \ref{KAlower}} \label{sec:KAlower}

The Agler class is called the \emph{Agler} class because
of special formulas established by Agler \cite{Agler} and later elucidated by
Agler-McCarthy \cite{AMpick} and Ball-Trent \cite{BTpick}.
It is shown in \cite{AMbook}, that functions $f \in \mA_d$ 
possess an \emph{isometric transfer function realization}:
\[
f(z) = A + B P(z) (I - D P(z))^{-1} C
\]
where $V = \begin{pmatrix} A & B \\ C& D\end{pmatrix}$ is a block
isometric operator on $\C\oplus \mH$.  Here $\mH$ is some Hilbert
space with a further orthogonal decomposition
\[
\mH = \mH_1\oplus \cdots \oplus \mH_d
\]
and $P(z) = \sum_{j=1}^{d} z_j P_j$ where $P_j$ is projection onto
$\mH_j$.  Note that since $V$ is a contraction, $|B|^2, |C|^2 \leq (1-|A|^2)$.
In particular, $A = f_0$ and
\begin{equation} \label{BCineq} 
|B||C| \leq (1-|f_0|^2).
\end{equation}

We can write the power series for $f$ in terms of the above data as
follows.  Let $[d] = \{1,\dots, d\}$.  Let 
\begin{equation}\label{words}
W(k,d) = \{w:[k] \to [d]\}
\end{equation}
be the set of functions from $[k]$ to $[d]$ (or ``words of length'' $k$ with $d$ letters).  
We define for $w \in W(k,d)$
\begin{equation}\label{alw}
\al(w) = (\#(w^{-1}(1)), \dots, \#(w^{-1}(d)) \in \N^d.
\end{equation}

For example, the word $w = 3221321 \in W(7,3)$ has $\al(w) = (2,3,2)$.

Expanding the transfer function realization for $f(z) = \sum f_{\al}
z^{\al}$ yields that for $k\geq 1$, the $k$-th homogeneous term is given by
\[
BP(z)(DP(z))^{k-1} C.
\]
Expanding this we obtain formulas for coefficients
\[
f_{\al} = \sum_{w \in W(k,d), \al(w) = \al} B P_{w(1)} D \cdots D
P_{w(k)} C.
\]
The main quantity we need to estimate is
\[
S_k = \sum_{\al: |\al| = k} |f_{\al}|.
\]
\begin{lemma} With $f$ as written above, for $k\geq 1$ we have
\[
S_k \leq (1-|f_0|^2)\binom{d+k-2}{k-1}^{1/2}.
\]
\end{lemma}

\begin{proof}
For any multi-index $\al$, let $\mu_{\al} = \overline{\text{sgn} f_{\al}}$, so that $\mu_{\al} f_{\al} = |f_{\al}|$.

For $k=1$, the inequality
follows from the Schwarz lemma for the polydisk, however
it may be instructive to give the proof using the transfer
function formula.
We let $e_1,\dots, e_d$ be standard basis vectors of $\mathbb{N}^d$ or $\mathbb{C}^d$.
We have
\[
S_1 = \sum_{j=1}^{d} \mu_{e_j} f_{e_j} 
= \sum_{j=1}^d \mu_{e_j} BP_j C 
= B\left(\sum_{j=1}^{d} \mu_{e_j} P_j\right) C.
\]
The operator $\sum_{j=1}^{d} \mu_{e_j} P_j$ is a partial
isometry (in fact, a unitary if none of the $\mu_{e_j}$'s vanish).
Therefore, $S_1 \leq |B||C| \leq (1-|f_0|^2)$.

For $k\geq 2$,
\[
\begin{aligned}
S_k &= \sum_{|\al|=k} \mu_{\al} f_{\al} \\
&= \sum_{|\al|=k} \mu_{\al} \sum_{w \in W(k,d), \al(w) = \al} B P_{w(1)} D \cdots D
P_{w(k)} C \\
&= \sum_{w\in W(k,d)} B P_{w(1)} D \cdots D \mu_{\al(w)}
P_{w(k)} C \\
&= \sum_{v\in W(k-1,d)} B P_{v(1)} D \cdots P_{v(k-1)} D \sum_{j=1}^{d} \mu_{\al(v)+e_j} P_j C\\
&= \sum_{v\in W(k-1,d)} B P_{v(1)} D \cdots P_{v(k-1)} D M_{\al(v)} C
\end{aligned}
\]
where $M_{\al} = \sum_{j=1}^{d} \mu_{\al+e_j} P_{j}$ is a partial isometry.  

This expression for $S_k$ can be viewed as a pairing of
coefficients of the two expressions
\[
B(P(z) D)^{k-1} 
=  \sum_{v \in W(k-1,d)} BP_{v(1)} D \cdots P_{v(k-1)} D z_{v(1)} \cdots z_{v(k-1)}
\text{ and }
\sum_{|\alpha| = k-1}  M_{\alpha} C z^{\alpha}
\]
and such a pairing can be accomplished with an integral expression:
\[
S_k = 
\int_{\T^{d-1}} B (P(z) D)^{k-1} \sum_{|\al|=k-1} M_{\alpha}  z^{-\alpha} C d\sigma(z)
\]
where $d\sigma$ is normalized Lebesgue measure on the torus $\T^{d-1}$.
Therefore, by Cauchy-Schwarz
\[
\begin{aligned}
S_k^2 
& \leq 
\int_{\T^{d-1}} |B (P(z)D)^{k-1}|^2 d\sigma \int_{\T^{d-1}} \left|\sum_{|\al|=k-1} M_{\alpha}  z^{-\alpha} C \right|^2 d\sigma
\leq |B|^2|C|^2 \binom{d+k-2}{k-1}. 
\end{aligned} 
\]
where the binomial coefficient is the count of $\al$ with $|\al| = k-1$.
So, $S_k \leq |B||C| \binom{d+k-2}{k-1}^{1/2}\leq (1-|f_0|^2)\binom{d+k-2}{k-1}^{1/2}$
by \eqref{BCineq}.
\end{proof}

To prove Theorem \ref{KAlower}, we estimate
\[
\|f_r\|_1 = \sum_{k=0}^{\infty} S_k r^k \leq  |f_0| + (1-|f_0|^2) \sum_{k=1}^{\infty} r^k  \binom{d+k-2}{k-1}^{1/2}.
\]
This is at most $1$ for all $|f_0| \in [0,1]$ if and only if
\[
\sum_{k=1}^{\infty} r^k  \binom{d+k-2}{k-1}^{1/2} \leq 1/2.
\]
Thus, if $r \in (0,1)$ is the root of 
\[
\sum_{k=1}^{\infty} r^k  \binom{d+k-2}{k-1}^{1/2} = 1/2
\]
then $\|f_r\|_1 \leq 1$ for all $f \in \mA_d$ meaning $r \leq K(\mA_d)$.

Since $\binom{d+k-2}{k-1}^{1/2} \leq d^{(k-1)/2}$, the solution to 
\[
\sum_{k=1}^{\infty} r^k d^{(k-1)/2}   = 1/2
\]
will be strictly smaller; however here we can readily solve for 
\[
r = \frac{1}{\sqrt{d}+2}.
\]
This proves $K(\mA_d) \geq \frac{1}{\sqrt{d}+2}$ and 
completes the proof of Theorem \ref{KAlower}.

\section{Lower bound for Schur-Agler radius: Theorem \ref{SAlower}}

To prove Theorem \ref{SAlower}, our main tool is a result from 
Grinshpan et al. \cite{grinshpan} that we can essentially use
as a ``black box'' except for a minor lemma.
Given a $d$-tuple of operators $T = (T_1,\dots, T_d)$ 
and a multi-index consider 
\[
\Delta_T^{\alpha}  = \sum_{0\leq \beta\leq \alpha} (-1)^{|\beta|} T^{\beta} (T^{\beta})^*.
\]
Here we are using the component-wise partial order on multi-indices $\beta \leq \al$.
For example, if $\al = (1,2)$, then 
\[
\Delta_T^{(1,2)} = I - T_1 T_1^* - T_2T_2^* + T_1T_2 (T_1T_2)^* + T_2^2 (T_2^2)^* - T_1T_2^2(T_1T_2^2)^*.
\]
We say $T$ belongs to the class $\mathcal{P}_{p,q}$ if $T$ is a $d$-tuple of 
commuting contractive operators such that
\[
\Delta_T^{\alpha} \geq 0
\]
for $\alpha = \sum_{j\ne p} e_j$ and $\alpha = \sum_{j\ne q} e_j$.

\begin{theorem}[\cite{grinshpan} Theorem 1.3]
If $T \in \mathcal{P}_{p,q}$ and $T$
consists of strict contractions, then $\|f(T)\| \leq 1$ for $f \in \mathcal{S}_d$.
\end{theorem}

The condition that $T$ belongs to $\mathcal{P}_{p,q}$ is implied by a simpler and more quantitative condition.

\begin{lemma}
If $S \subseteq \{1,\dots, d\}$ and $\alpha = \sum_{j\in S} e_j$ then
\[
\sum_{j\in S} T_{j} T_j^* \leq I
\]
implies $\Delta_T^{\alpha} \geq 0$.
In particular, if $\|T_j\| \leq 1/\sqrt{\#S}$ for $j\in S$ then $\Delta_T^{\alpha} \geq 0$.
\end{lemma}

\begin{proof}
There is no harm in assuming $S = \{1,\dots, k\}$.  Set $v_k = \sum_{j=1}^{k} e_j$.
We prove by induction on $k$ that $I-\sum_{j=1}^{k} T_j (T_j)^* \leq \Delta_T^{v_k} \leq I$.
For $k=1$, the claim is immediate.
Next, assuming the claim for $k$, we have
\[
\Delta_{T}^{v_{k+1}} = \Delta_{T}^{v_k} - T_{k+1} \Delta_{T}^{v_k} (T_{k+1})^* 
\geq I - \sum_{j=1}^{k} T_j(T_j)^* - T_{k+1} (T_{k+1})^* 
\]
and because $\Delta_{T}^{v_k} \geq 0$
\[
\Delta_{T}^{v_{k+1}} \leq I.
\]
\end{proof}

To prove Theorem \ref{SAlower} we simply remark that for a $d$-tuple $T$ of contractions
\[
\sum_{j\ne p} \frac{1}{d-1} T_j (T_j)^* \leq I
\]
implying that for $\tilde{T} = \frac{1}{\sqrt{d-1}} T$, 
$\Delta^{\al}_{\tilde{T}} \geq 0$ for $\al = \sum_{j\ne p} e_j$.
A similar conclusion holds for $\al = \sum_{j\ne q} e_j$.
Thus, for $r = \frac{1}{\sqrt{d-1}}$ and $f \in \Sz_d$ we have
\[
\| f_r(T)\| = \|f(\tilde{T})\| \leq 1.
\]
Thus, $SA_d \geq \frac{1}{\sqrt{d-1}}$, completing the proof.

Note that \cite{drexelnorm} constructs a homogeneous polynomial $p$ in $3$ variables
with degree $2$ and ratio $\frac{\|p\|_A}{\|p\|_\infty} \geq 1.23$ implying
\[
SA_3 \leq \frac{1}{\sqrt{1.23}} \leq 0.902
\]
which we compare to the lower bound $SA_3 \geq \frac{1}{\sqrt{2}} \geq 0.707$.
See Remark 6.2 of \cite{drexelnorm}.

\section{Upper bound for Bohr-Agler radius: Theorem \ref{KAupper}}

To prove Theorem \ref{KAupper}, we use a construction
essentially contained in Maurizi-Queffelec \cite{MQ}.
Let $q,m$ be positive integers.
Let $A_q$ be the $q\times q$ matrix $A_q = (\omega^{ij})_{i,j =0,\dots, q-1}$
where $\omega = \exp(2\pi i/q)$.  Then,
\[
A_q^* A_q  = q I \text{ and } \|A_q\| = \sqrt{q} \text{ (using operator norm)}.
\]
Define $q$-tuples of variables $z^{(j)} = (z_1^{(j)},\dots, z_q^{(j)})$ for $j=1,\dots, m$.
Let $\Delta(z_1,\dots, z_q)$ be the matrix with $z_1,\dots, z_q$ on the diagonal.
Define
\begin{equation} \label{Pdef}
P(z^{(1)},\dots, z^{(m)}) = e_1^t A_q \Delta(z^{(1)}) \cdots A_q \Delta(z^{(m)}) \mathbbm{1}
\end{equation}
where $e_1 = (1,0,\dots ,0)^t$ and $\mathbbm{1} \in \C^q$ is the all ones vector.

\begin{lemma} 
The polynomial $P$ from \eqref{Pdef} is homogeneous of degree $m$ and satisfies
\[
\|P\|_1 = q^m \text{ and } \|P\|_{A} \leq q^{(m+1)/2}.
\]
\end{lemma}

Recall the coefficient-wise $\ell^1$ norm $\|\cdot\|_1$ from \eqref{ell1} and the Agler norm $\|\cdot\|_{A}$
from  \eqref{Aglernorm}.

\begin{proof}
Denote the entries of $A_q$ by $A_{ij}$.
Fully expanding $P(z)$ we have 
\[
\sum_{i_1,\dots, i_m} A_{1i_1}z^{(1)}_{i_1} A_{i_1,i_2} z^{(2)}_{i_2} \cdots A_{i_{m-1}, i_{m}} z^{(m)}_{i_m} 
\]
and since the entries of $A$ are unimodular we see that all nonzero coefficients of $P$ are unimodular.
There are $q^m$ such coefficients so we have $\|P\|_1 = q^m$.

To estimate $\|P\|_{A}$ we can essentially insert an operator tuple into $P$'s formula in block form.
For an operator $qm$-tuple, $T = (T^{(1)}, \dots, T^{(m)})$ where each $T^{(j)}$ is itself a $q$-tuple
of operators on a Hilbert space $\mathcal{H}$ we have 
\[
P(T) =  (e_1^t\otimes I) (A_q\otimes I) \Delta(T^{(1)})\cdots (A_q\otimes I) \Delta(T^{(m)}) (\mathbbm{1}\otimes I)
\]
where $\Delta(T^{(j)})$ is the block diagonal operator with $T_1^{(j)},\dots, T_q^{(j)}$ on the diagonal.
More formally, we have $\Delta(T^{(j)}) := \sum_{i=1}^{q} \Delta(e_i)\otimes T_i^{(j)}$.
This formula could be verified by simply multiplying out and checking that it agrees with the fully
expanded form for $P$.
Since $\|A_q\|=\sqrt{q}$, $\|\mathbbm{1}\|=\sqrt{q}$ and $\|\Delta(T^{(j)})\| \leq 1$
we see that
\[
\|P(T)\| \leq q^{(m+1)/2}
\]
since $A_q$ occurs $m$ times and $\mathbbm{1}$ occurs once.
This proves $\|P\|_{A} \leq q^{(m+1)/2}$.
\end{proof}

To prove Theorem \ref{KAupper}, we take the $P$ constructed above and
estimate
\[
\frac{\| P_r\|_1}{\|P\|_A} \geq \frac{r^m q^m}{q^{(m+1)/2}} = r^m q^{(m-1)/2}.
\]
If this is greater than $1$, then $r > K(\mA_{qm})$.
So, $r > q^{(1-m)/2m}$ implies $r > K(\mA_{qm})$.  Thus,
$q^{\frac{1}{2m} - \frac{1}{2}} \geq K(\mA_{qm})$.

Since $K(\mA_d)$ is decreasing with respect to $d$, 
we can take $m= \lfloor \log d\rfloor$, $q = \left\lfloor \frac{d}{\log d} \right\rfloor$
to bound
\[
K(\mA_{qm}) \geq K(\mA_d).
\]
{\color{red}
We can write 
\[
q 
=\frac{d}{\log d}\left(1 +O\left( \frac{\log d}{d}\right)\right)
\quad\text{ and }\quad
m 
=
\log d\left( 1+ O\left(\frac{1}{\log d}\right)\right)
\]
and proceed to estimate $\left(\frac{1}{2m} - \frac{1}{2}\right) \log q$.
We have
\[
\log q 
=
\log d - \log\log d + O\left(\frac{\log d}{d}\right)
\]
and
\[
\frac{1}{2m} - \frac{1}{2}
= -\frac{1}{2} + \frac{1}{2 \log d} + O\left(\frac{1}{(\log d)^2}\right)
\]
so that
\begin{align*}
\left(\frac{1}{2m} - \frac{1}{2}\right) \log q
&= \frac{1}{2} + \frac{1}{2}\log\left(\frac{\log d}{d}\right)
 - \frac{\log\log d}{2 \log d} + O\left(\frac{1}{\log d}\right)\\
& \leq
 \frac{1}{2} + \frac{1}{2}\log\left(\frac{\log d}{d}\right)
 + O\left(\frac{1}{\log d}\right)
\end{align*}

and therefore
\[
q^{\frac{1}{2m} -\frac{1}{2}}
\leq
e^{1/2} \sqrt{\frac{\log d}{d}} \exp\left(O\left(\frac{1}{\log d}\right)\right).
\]

Since $K(\mA_d) \geq K_d \sim \sqrt{\frac{\log d}{d}}$ we see that 
\[
1 \leq \liminf_{d\to \infty} \frac{K(\mA_d)}{\sqrt{\frac{\log d }{d}}} 
\leq \limsup_{d\to \infty} \frac{K(\mA_d)}{\sqrt{\frac{\log d }{d}}} \leq e^{1/2}.
\]
This proves Theorem \ref{KAupper}.
}

%
%

\section{Upper bound for Schur-Agler radius: Theorem \ref{SAupper}}

Recalling \eqref{Ckd}, Dixon obtained upper and lower bounds for $C_k(d)$.
The lower bound stated that there exists a constant $c_k$ such that
\begin{equation} \label{Ckdlower}
C_{k}(d) \geq c_k d^{\frac{1}{2}\lfloor \frac{k-1}{2}\rfloor}.
\end{equation}
See Theorem 1.2(b) of \cite{dixon}.
We will discuss the upper bound in the next section.
In order to estimate $SA_d$ we need to control the constant $c_k$ and subsequently
need to dig into the proofs in \cite{dixon}.
The proof of \eqref{Ckdlower} used the Kahane-Salem-Zygmund inequality
and a lemma about partial Steiner systems.

To define partial Steiner systems, let $[d]=\{1,\dots, d\}$ and call
a subset of $[d]$ of cardinality $k$, a $k$-subset.
Given integers $1\leq t\leq k \leq d$, a \emph{partial Steiner system} $S_p(t,k,d)$ is a 
collection of $k$-subsets $A_1,\dots, A_N \subset [d]$ such that
any $t$-subset of $[d]$ is contained in at most one of the sets $A_j$ for $j=1,\dots, N$.
The largest $N$ such that there exists a partial Steiner system $S_p(t,k,d)$ will be 
denoted $N(t,k,d)$.  By a pigeonhole argument
\[
N(t,k,d) \leq \frac{\binom{d}{t}}{\binom{k}{t}}.
\]
Lemma (3.2) of \cite{dixon} proves that for fixed $k,t$, 
there exists a constant $a_{t,k}$ such that $N(t,k,d) \geq a_{t,k} d^{t}$.
Estimates on $N(t,k,d)$ have improved dramatically over the years, however
they typically all estimate $N(t,k,d)$ for $t,k$ held \emph{fixed} and $d\to \infty$.
For instance, the result of Rödl \cite{rodl} states that for fixed $t,k$
\[
N(t,k,d)  = (1- o(1))\frac{\binom{d}{t}}{\binom{k}{t}}
\]
where the $o(1)$ term (again for fixed $t,k$) goes to $0$ as $d\to \infty$.  
For our purposes we need control over $N(t,k,d)$ for $t, k \approx \log d$.
It is not our aim to get into the combinatorics of partial Steiner systems,
so it is fortunate that the proof of Lemma (3.2) in \cite{dixon} provides the explicit estimate
\[
N(t,k,d) \geq \binom{d}{k} \binom{d}{k-t}^{-1} \binom{k}{t}^{-1}
\]
(from which the $a_{t,k} d^t$ lower bound is directly derived).
We only need to know the following crude bound which holds for $d\geq 2k$
\begin{equation} \label{Nlower}
N(t,k,d) \geq  \binom{d}{k} \binom{d}{k-t}^{-1} \binom{k}{t}^{-1} \geq 4^{-k} \left(\frac{d}{k}\right)^{t}.
\end{equation}
 
To see this, we note for $d \geq 2k \geq 2t$
\[
\begin{aligned}
\binom{d}{k} \binom{d}{k-t}^{-1} \binom{k}{t}^{-1}
&=
\left(\frac{d}{k}\right)^t \left(1- \frac{k-t}{d}\right)\cdots \left(1- \frac{k-1}{d}\right) \binom{k}{t}^{-1}\\
 &\geq \left(\frac{d}{k}\right)^t 2^{-t-k} \geq \left(\frac{d}{k}\right)^t 4^{-k}.
\end{aligned}
\]

Next, we discuss how $N(t,k,d)$ is involved in the proof of \eqref{Ckdlower}.  
Using the Kahane-Salem-Zygmund inequality and the existence
of a partial Steiner systems, \cite{dixon} proves there exists a
homogeneous polynomial $p$ in $d$ variables of odd degree $k=2n+1$ such 
that
\[
\frac{\|p\|_A}{\|p\|_{\infty}} > \frac{1}{8} \left( \frac{N(n+1, k, d)}{d \log k} \right)^{1/2}.
\]
This is directly pulled from the proof of Theorem 1.2(b) in \cite{dixon}.
For $d \geq 2k$ we have from \eqref{Nlower}
\[
\frac{\|p\|_A}{\|p\|_{\infty}} > \frac{1}{8} \left( 4^{-k} \frac{(d/k)^{n+1}}{d \log k} \right)^{1/2} 
= \frac{2^{-k}}{8} \frac{d^{n/2}}{k^{(n+1)/2} \sqrt{\log k}}.
\]

We can now complete the proof of Theorem \ref{SAupper}.
With $p$ as above $\| p_r\|_A/\|p\|_{\infty} = r^k\| p\|_A/\|p\|_{\infty} > 1$ provided
\begin{equation} \label{provided}
r > (2) 8^{1/k} \frac{ (\log k)^{1/(2k)} k^{(n+1)/(2k)}}{d^{n/2k}}.
\end{equation}
Note $k^{(n+1)/(2k)} = \exp( \frac{n+1}{2(2n+1)} \log (2n+1)) = 
2^{1/4} \exp((1/4)\log n + O(\frac{\log n}{n}))$.
Thus, $k^{(n+1)/(2k)}  \lesssim k^{1/4}$.
Note also $d^{n/2k} = \exp(\frac{n}{2(2n+1)}\log d)$ and 
$\frac{n}{2(2n+1)} = \frac{1}{4}( 1+ O(1/n))$.
If we choose $n$ comparable to $\log d$ then
$d^{n/2k} \gtrsim d^{1/4}$.  
Thus, for $n$ comparable to $\log d$, there is a constant $c$ such that the right side of 
\eqref{provided} is bounded by
\[
c \left(\frac{\log d}{d}\right)^{1/4}.
\]
So, for $r > c \left(\frac{\log d}{d}\right)^{1/4}$ we have $r > SA_d$.
Therefore, $SA_d \leq c \left(\frac{\log d}{d}\right)^{1/4}$ proving
Theorem \ref{SAupper}.

\section{Improvements to Dixon's bounds: Theorem \ref{improvedixon}}\label{sec:dixon}

Recall the constant $C_k(d)$ defined in \eqref{Ckd}.
As mentioned above, the paper \cite{dixon} contains bounds on $C_{k}(d)$
that for fixed $k$ have the correct growth with respect to $d$; namely $d^{\frac{k-2}{2}}$.  
Specifically, Dixon proves
\[
C_k(d) \leq G_{\C} (3d)^{(k-2)/2} (2e)^k
\]
where $G_{\C} \in [1.33807, 1.40491]$ is the complex Grothendieck constant (see \cite{grot}).
Using modern (and not so modern) 
improvements to the various lemmas contained in \cite{dixon}
we can improve the bound.

The two main improvements that can be introduced to the proof
are better Bohnenblust-Hille bounds for multilinear forms and
polarization constants.  We also hope that our exposition will
be of some use to the literature as \cite{dixon} and the related
paper \cite{davie} freely and implicitly use dual versions of inequalities
that are more easily stated in their original forms.  A good general
reference for what follows is \cite{Defantbook} (especially Section 6.4 and Chapter 16).  

First, the Bohnenblust-Hille inequality \cite{BH} states there is a (smallest) constant
$B_{m}$
such that for any $d$ and any $m$-multilinear form 
$L:  \C^d\times \cdots \C^d \to \C$
\[
\left(\sum_{i_1,\dots,i_m=1}^{d} |L(e(i_1),\dots, e(i_m))|^{\frac{2m}{m+1}} \right)^{\frac{m+1}{2m}}
\leq B_m \|L\|_\infty
\]
where
\[
\|L\|_{\infty} = \sup\{|L(z^{(1)},\dots, z^{(m)})|: z^{(j)} \in \D^d \text{ for } j =1,\dots, m\}
\]
and $e(i)$ represents a standard basis vector for $\C^d$---these
were earlier denoted by $e_i$ but in the interest of reducing subscripts we
use this notation here.  
There is a long and interesting 
history of the Bohnenblust-Hille inequality and
improvements on estimates for $B_m$ for which 
we refer to \cite{defant}, \cite{bayart}.
For our purposes we shall quote the best result we know of.
We caution that the paper \cite{bayart} uses a more general
notation $B^{mult}_{\C,m}$ for $B_m$ since they are interested
in a variety of related constants.

\begin{theorem}[Corollary 3.2 of \cite{bayart}] \label{Bm}
There exists a constant $\kappa >0$ such that for $m\geq 1$
\[
B_m \leq \kappa m^{\frac{1-\gamma}{2}}
\]
where $\gamma$ is the Euler-Mascheroni constant
and $\frac{1-\gamma}{2} \approx 0.21392$.
\end{theorem}

Next, the polarization inequality states that there
is a constant $Pol_m$ such that for any $d$ and for
any symmetric $m$-multilinear form $L:  \C^d\times \cdots \C^d \to \C$
we have 
\[
\|L\|_{\infty} \leq Pol_{m} \sup\{|L(z,\dots, z)|: z \in \D^d\}.
\]
Usually, this is stated in terms of starting with an $m$-homogeneous
polynomial $P \in \C[z_1,\dots, z_d]$ 
\[
P(z) = \sum_{|\al|=m} P_{\al} z^{\al}.
\]
We can write the monomials in terms of ``words'' using the notations
$W(m,d)$ (see \eqref{words}) and $\al(w)$ (see \eqref{alw}) 
from Section \ref{sec:KAlower}
\[
P(z) = \sum_{w \in W(m,d)} \frac{P_{\al(w)}}{\binom{m}{\al(w)}} z_w
\]
where $z_w = z_{w(1)}\cdots z_{w(m)}$.
(The number of words $w$ with $\al(w) = \al$ is $\binom{m}{\al}$.)
Then,
\[
Sym(P)(z^{(1)},\dots, z^{(m)}) := \sum_{w \in W(m,d)} \frac{P_{\al(w)}}{\binom{m}{\al(w)}} z^{(1)}_{w(1)} \cdots z^{(m)}_{w(m)}
\]
is a symmetric $m$-multilinear form with $Sym(P)(z,\dots ,z) = P(z)$.

A theorem of Harris gives a bound on $Pol_m$.  See also Dineen \cite{dineen} (Propositions 1.43, 1.44).

\begin{theorem}[Corollary 5 of \cite{harris}] \label{polar}
\[
Pol_m \leq \frac{m^{\frac{m}{2}} (m+1)^{\frac{m+1}{2}}}{2^m m!}.
\]
\end{theorem}

Next we need to give a dual version of the Bohnenblust-Hille inequality.
The projective tensor norm on the space of $m$-multilinear 
forms on $\C^d$ is given by
\[
\|L\|_V = \inf\{ \sum_{r=1}^{N} |\lambda_r|\}
\]
where the infimum is taken over all possible representations
of 
\[
L = \sum_{r=1}^{N} \lambda_r f^{(1)}_r\otimes \cdots \otimes f^{(m)}_r
\]
where each $f^{(j)}_r: \C^d \to \C$ is a linear functional with
$\|f^{(j)}_r\|_{\infty} = \sup_{z\in \D^d} |f^{j}_r(z)| \leq 1$.
Here
\[
(f^{(1)}_r\otimes \cdots \otimes f^{(m)}_r)(z^{(1)},\dots, z^{(m)})
 = 
 f^{(1)}_r(z^{(1)}) \cdots  f^{(m)}_r(z^{(m)}).
 \]
 The explicit form of $\|L\|_V$ will only be used once; more importantly,
 this norm is dual to the $\infty$ norm.
 Let us write $e(w):= e(w(1)) \otimes\cdots \otimes e(w(m))$ and
 $L_w = L(e_w)$ for short so that 
 $L = \sum_{w\in W(m,d)} L_w e(w)$.
 Then,
  \[
 \|L\|_V = \sup\{ |\sum_{w \in W(m,d)} L_w F_w|\}
\]
with the supremum taken over $m$-multilinear $F = \sum_w F_w e(w)$
such that $\|F\|_{\infty} \leq 1$.  
The Bohnenblust-Hille inequality implies
that for such $F$,
$(\sum_{w} |F_w|^{\frac{2m}{m+1}})^{\frac{m+1}{2m}} \leq B_m$, and therefore
by H\"older's inequality
\begin{equation} \label{LV}
\|L\|_V \leq B_m \left(\sum_{w} |L_w|^{\frac{2m}{m-1}}\right)^{\frac{m-1}{2m}}
\leq
B_m d^{\frac{m-1}{2}} \sup_w |L_w|
\end{equation}
since $\#W(m,d) = d^m$.

Let us now rehash Dixon's proof and insert the appropriate estimates.

\begin{proof}[Proof of Theorem \ref{improvedixon}]
To begin let $P \in \C[z_1,\dots, z_d]$ be $k$-homogeneous,
let $T$ be a $d$-tuple of strictly commuting contractive operators
on a Hilbert space $\mathcal{H}$ and let $x,y \in \mathcal{H}$ 
be unit vectors.
As above we write $P(z) = \sum_{\al} P_{\al} z^{\al}$ in symmetric 
``word'' form
\[
P(z) = \sum_{w \in W(k,d)} L_w z_w
\]
with $L_w = \frac{P_{\al(w)}}{\binom{k}{\al(w)}}$.
Set $L = Sym(P)$.
Then,
\[
\langle P(T)x,y \rangle =
\sum_{w\in W(k,d)} L_w \langle T_w x,y\rangle
\]
and separating $w$ into $v\in W(k-1,d)$, $j\in [d]$
this equals
\[
\sum_{v\in W(k-1,d),j\in [d]} L_{v,j} \langle T_v x, T_j^*y\rangle.
\]
By Grothendieck's inequality (see \cite{grot}), this is bounded by
\[
G_{\C} \sup \left\{ \left|\sum_{v\in W(k-1,d), j\in [d]} L_{v,j} G_v h_j\right|:
|G_v|,|h_j|\leq 1 \text{ for all } v,j \right\}.
\]
For $G = \sum_{v\in W(k-1,d)} G_v e(v)$ with $|G_v|\leq 1$ 
we have $\|G\|_V \leq B_{k-1} d^{\frac{k-2}{2}}$ by \eqref{LV}.
Given such a $G$, for any $\epsilon>0$ we can write
\[
G = \sum_{r=1}^{N} \lambda_r f^{(1)}_r \otimes \cdots f^{(k-1)}_{r}
\]
with $\sum_{r=1}^{N} |\lambda_r| \leq \|G\|_V +\epsilon$
and $\|f^{(j)}_{r}\|_{\infty}\leq 1$.
Then, because pairing the coefficients
of $L$ up with the coefficients of simple tensors whose 
coefficients are bounded by 1 amounts to evaluating $L$ on $(\D^d)^m$ 
we have 
\[
\left|\sum_{v\in W(k-1,d), j\in [d]} L_{v,j} G_v h_j \right|
\leq \sum_{r=1}^{N}|\lambda_r| \|L\|_{\infty} 
\leq (\|G\|_{V}+\epsilon)\|L\|_{\infty}
\leq (B_{k-1} d^{\frac{k-2}{2}} +\epsilon) \|L\|_{\infty}.
\]
Letting $\epsilon \to 0$,
\[
|\langle P(T)x,y \rangle | \leq G_{\C} B_{k-1} d^{\frac{k-2}{2}} \|(Sym(P)\|_{\infty}.
\]
Finally, using polarization $\|P\|_A \leq G_{\C} B_{k-1} d^{\frac{k-2}{2}} Pol_{k} \|P\|_{\infty}.$

Thus, $C_{k}(d) \leq G_{\C} B_{k-1} d^{\frac{k-2}{2}} Pol_{k}$.
Inserting explicit estimates (Theorems \ref{Bm} and \ref{polar})
\[
C_{k}(d) \lesssim (k-1)^{\frac{1-\gamma}{2}} d^{\frac{k-2}{2}} 
\frac{k^{\frac{k}{2}} (k+1)^{\frac{k+1}{2}}}{2^k k!}
\]
and using a crude Stirling bound $k! \gtrsim \sqrt{k} (k/e)^k$
yields $C_k(d) \lesssim (e/2)^k d^{\frac{k-2}{2}} k^{\frac{1-\gamma}{2}}$
which implies the claimed inequality.
\end{proof}

One could of course replace $k^1$ with $k^{\frac{1-\gamma}{2}}$ in our
statement of Theorem \ref{improvedixon}, but we wrote $k^1$ for
simplicity.

It is worth stating explicitly---in anticipation of potential improvements of
the various constants---the following corollary of the proof.
\begin{corollary}
\[
C_{k}(d) \leq G_{\C} B_{k-1} d^{\frac{k-2}{2}} Pol_{k}.
\]
\end{corollary}

We have gone through improvements to Dixon's bounds
partially as a cautionary tale that they are not sufficient
to improve the lower bound $SA_d \geq K_d$.
As far as we can tell, the Dixon bound is only sufficient for
proving $SA_d \gtrsim c d^{-1/2}$ which
we have already obtained more explicitly.

Finally, we would like to draw attention to a
related inequality recently established by Hartz \cite{hartz}.
In particular,  \cite{hartz} proves $C_k(d) \lesssim_d (\log(k+1))^{d-3}$
however since the implicit constant depends on $d$
we cannot directly apply this result here to study $SA_d$.

\end{document}